\documentclass[12pt,reqno]{amsart}
\usepackage{amssymb}
\usepackage{txfonts}
\usepackage{amsmath}

 \newtheorem{thm}{Theorem}[section]
 
 \newtheorem{lem}[thm]{Lemma}

 \theoremstyle{definition}
 \newtheorem{defn}[thm]{Definition}
 \theoremstyle{remark}
 \newtheorem{rem}[thm]{Remark}
 \numberwithin{equation}{section}

\begin{document}

\title[]
{On the Perelman's reduced entropy and Ricci flat manifolds with maximal volume growth}

\author{Liang Cheng, Anqiang Zhu}

\dedicatory{}
\date{}

 \subjclass[2000]{
Primary 53C44; Secondary 53C42, 57M50.}

\keywords{Ricci flat manifolds; Maximal volume growth; Quadratic curvature decay; Ricci flow;  Perelman's reduced entropy}
\thanks{}

\address{Liang Cheng, School of Mathematics and Statistics, Huazhong Normal University,
Wuhan, 430079, P.R. CHINA}

\email{math.chengliang@gmail.com}

\address{Anqiang Zhu, School of Mathematics and Statistics, Wuhan University,
Wuhan, 430072, P.R. CHINA}

\email{anqiangzhu@yahoo.com.cn}

 \maketitle

\begin{abstract}
In this paper, we study the Ricci flat manifolds with maximal volume growth using Perelman's reduced volume of Ricci flow. We show that if $(M^n,g)$ is an noncompact complete Ricci flat manifold with maximal volume growth satisfying $|Rm|(x)\to 0$ as $d(x)=d_g(x,p)\to \infty$, then $M^n$ has the quadratic curvature decay. Some applications to this result are also presented.
\end{abstract}

\section{Introduction}
Let $(M^n,g)$ be a complete $n$-dimensional manifold. Denote $Rc$ be the Ricci curvature of $(M^n,g)$. The manifold $(M^n,g)$ is call the Ricci flat manifold if it satisfies $Rc \equiv 0$.
It has been an important question to study Ricci flat manifolds both in Riemannian geometry and physics.
 In physics, the Ricci flat manifolds represent vacuum solutions to the analogues of Einstein's equations for Riemannian manifolds of any dimension, with vanishing cosmological constant.
The Calabi-Yau manifolds, which are a special case of Ricci flat manifolds, play the important role
  in certain branches of mathematics such as algebraic geometry, as well as in theoretical physics. Especially in superstring theory, the extra dimensions of spacetime are conjectured to take the form of a 6-dimensional Calabi-Yau manifold.
We focus our attention on noncompact Ricci flat manifolds in this paper. There are many examples of Ricci flat manifolds and one may see \cite{GH}, \cite{GNP}, \cite{H}, \cite{K}, \cite{Pe1}, \cite{Pe2}, \cite{TY} and the references therein for more information.

\begin{defn}
We say the noncompact complete manifold $(M^n,g)$ has quadratic curvature decay if
\begin{equation}\label{condition_1}
|Rm|(x)\leq \frac{C}{d^2(x)},
\end{equation}
 where $|Rm|$ is the norm of curvature tensor $Rm$, $C$ is a finite constant and $d(x)=d_g(p,x)$ for some fixed point $p\in M^n$.
\end{defn}

The topology of the manifolds with quadratically curvature decay was also studied by many geometrists. It should be point out first that quadratic decay of curvature has no topological meaning since M.Gromov \cite{G} showed that any connected smooth paracompact manifold has a complete metric with quadratic curvature decay (also see \cite{L2}). However, an interesting result proved by U.Abresh \cite{A} that a complete manifold has finite topological type if its curvature decays faster than quadratically, i.e $|Rm|(x)=O(d(x)^{-2-\epsilon})$ for some $\epsilon>0$.
Here a manifold $M^n$ is said to have finite topological type if there is a compact domain $\Omega$ whose boundary $\partial\Omega$ is a topological manifold such that $M\backslash \Omega$ is homeomorphic to $\partial \Omega \times [0,\infty)$. Thus, there exists a critical rate of curvature decay, the quadratically curvature decay, around which the situation changes completely.
Then J.Sha and Z.Shen \cite{SS} proved that a complete manifold with quadratic curvature decay has the finite topology type if it has nonnegative Ricci curvature and maximal volume growth, where
we use the following
\begin{defn}
 We say the noncompact complete manifold $(M^n,g)$ has maximal volume growth if
\begin{equation}
\frac{\text{Vol}_gB(p,r)}{\omega_n r^n}\geq \alpha_M>0, \ \forall r>0,
\end{equation}
for some $p\in M^n$, where $\alpha_M$ is a constant and $\omega_n$ is the volume of unit ball in $\mathbb{R}^n$.
\end{defn}
\noindent In \cite{L2}, J.Lott and Z.Shen also obtained that
a complete manifold with quadratic curvature decay has the finite topology type if $vol_g(B(p,r)) =
o(r^2)$ as $t\to\infty$ and $M^n$ does not collapse at infinity, i.e. $\inf\limits_{x\in M} vol_g(B(x,1))>0$.

The aim of this paper is to study on what conditions can guarantee the Ricci flat manifolds have the quadratic curvature decay?
In \cite{CT}, J.Cheeger and G.Tian  proved that a $4$-dimensional noncompact complete Ricci flat manifold with its curvature in $L^2$ has the quadratic curvature decay. Another result in this direction is the following theorem, proved by S.Bando, A.Kasue, H.Nakajima \cite{BKN}: an $n$-dimensional noncompact complete Ricci flat manifold $M^n$ with maximal volume growth satisfying
\begin{equation}
\int_M |Rm|^{\frac{n}{2}} dvol_g<\infty
\end{equation}
 has faster than quadratic curvature decay, i.e $|Rm|(x)=O(d(x)^{-2-\epsilon})$ for some $\epsilon>0$. A similar result without assuming maximal volume growth also obtained by V.Minerbe \cite{M}.

Our main result in this paper is
the following
\begin{thm}\label{main}
Let $(M^n,g)$ be the noncompact complete Ricci flat manifold with maximal volume growth. If
\begin{equation}\label{condition_3}
|Rm|(x)\to 0, \ \text{as} \ d(x)=d_g(x,p)\to \infty,
\end{equation}
then
$M^n$ has the quadratic curvature decay.
\end{thm}

\begin{rem}
 The condition (\ref{condition_1}) is equivalent to
\begin{equation}\label{condition_2}
\int_{B(p,2r)\backslash B(p,r)} |Rm|^{\frac{n}{2}}\leq \Lambda,\ \text{$\Lambda$ is independent of $r$,}
\end{equation}
on the Ricci flat manifold $(M^n,g)$ with maximal volume growth (see \cite{CT1}).  So it follows from Theorem \ref{main} that
the conditions (\ref{condition_1}), (\ref{condition_3}) and (\ref{condition_2}) are equivalent to each other on Ricci flat manifolds with maximal volume growth.
\end{rem}

As the applications to Theorem \ref{main}, the condition (\ref{condition_1}) can be replaced by apparently weak condition (\ref{condition_3}) in some results for Ricci flat manifolds with maximal volume growth.
First, by Theorem \ref{main} and the result proved by J.Sha and Z.Shen (see Theorem 1.1 in \cite{SS}), we have the following

\begin{thm}\label{cor_1}
Let $(M^n,g)$ be an noncompact complete Ricci flat manifold with maximal volume growth. If $|Rm|(x)\to 0$ as $d(x)=d_g(x,p)\to \infty$, then
$(M^n,g)$ has the finite topological type.
\end{thm}

\begin{rem}
In \cite{AKL}, M.T.Anderson, P.B.Kronheimer, C.LeBrun constructed an example that there exists a complete Ricci-flat k\"{a}hler manifold which has the infinite topological type.
\end{rem}

In \cite{CT1}, J.Cheeger and G.Tian studied the cone structure at infinity of Ricci flat manifolds with maximal volume growth and quadratic curvature decay. Let $(M^n,g)$ be the manifold with nonnegative Ricci curvature and maximal volume growth. Then $(M^n,r_i^{-2}g, p_0)$ subconverges in the pointed Gromov-Hausdorff topology to a length space $M_{\infty}$ by Gromov's compactness theorem, where $p_0\in M^n$ and $r_i\to +\infty$. Moreover, $M_{\infty}=C(N^{n-1})$ is a metric cone \cite{CC}. Note that $M_{\infty}$ may not unique \cite{P2}. However, J.Cheeger and G.Tian \cite{CT1} proved that $M_{\infty}$ is unique if $(M^n,g)$ is a Ricci flat manifold with maximal volume growth and the cone $M_{\infty}$ is integrable (see Definition 0.11 in \cite{CT1}).
Combining Theorem \ref{main} and Theorem 0.13 in \cite{CT1}, we have the following theorem holds.

\begin{thm}\label{cor_2}
Let $(M^n,g)$ be an noncompact complete Ricci flat manifold with maximal volume growth. Fix any point $p_0\in M^n$. Suppose that $(M^n,r_i^{-2}g, p_0)$ subconverges in the pointed Gromov-Hausdorff topology to a metric cone $M_{\infty}=C(N^{n-1})$, where $r_i\to +\infty$. If $|Rm|(x)\to 0$ as $d(x)=d_g(x,p)\to \infty$ and the cone $M_{\infty}=C(N^{n-1})$ is integrable \cite{CT1}, then
$M_{\infty}$ is unique.
\end{thm}

The main tool for the proof of Theorem \ref{main} is the Perelman's reduced volume for the Ricci flow.
In \cite{P1}, Perelman introduced the reduced entropy (i.e. reduced distance and reduced volume), which becomes one of powerful tools for studying Ricci flow.
 Let $g(\tau)$ solves the backward Ricci flow
\begin{align}\label{backward_Ricci_flow}
    \frac{\partial g}{\partial \tau}=2Rc,
\end{align}
 on $M^n\times
[-T,0]$. Fix $p\in M^n$ and let $\gamma$ be a path
joining $(p,0)$ and $(y,\tau)$. Then the $\mathcal{L}$-length is defined in \cite{P1} as
\begin{align}
    \mathcal{L}_p(\gamma)=\int^{\tau}_0\sqrt{\eta}(R(\gamma(\eta))+|\gamma'(\eta)|^2)d\eta.
\end{align}
Denote $L_p(y,\tau)$ be the length of a shortest
$\mathcal{L}$-length joining $(p,0)$ and $(y,\tau)$. Let
$
l_p(y,\tau)=\frac{L_p(y,\tau)}{2\sqrt{\tau}}
$
be the $l$-length joining $(p,0)$ and $(y,\tau)$. Perelman's
reduced volume is defined as
\begin{align}\label{r_volume}
   \mathcal{V}_p(\tau)=\int_{M} (4\pi\tau)^{-\frac{n}{2}}e^{-l_p(y,\tau)}dvol_{g(\tau)}(y).
\end{align}
Perelman \cite{P1} showed that the reduced volume defined in (\ref{r_volume}) is
monotone non-increasing under the backward Ricci flow (\ref{backward_Ricci_flow}) and the Ricci flow is a gradient shrinking soliton if $ \mathcal{V}_p(\tau)$ is independent of $\tau$ (also see Theorem \ref{reduced_volume}). We remark that
\begin{align}\label{flat_reduced_l}
l_p(y,\tau)=\frac{d_g(p,y)^2}{4\tau},
\end{align}
and
\begin{align}\label{flat_r_volume}
   \mathcal{V}_p(\tau)=\int_{M} (4\pi\tau)^{-\frac{n}{2}}e^{-\frac{d_g(p,y)^2}{4\tau}}dvol_{g}(y),
\end{align}
for the static Ricci flow on the Ricci flat manifold $(M^n,g)$.

At first sight, it seems difficult to study the geometric and topological properties of Ricci flat manifold by using Ricci flow since the Ricci flow is static on such manifold.
However, since the Ricci flat shrinking soliton must be isometric to $\mathbb{R}^n$ (see Lemma 7.1 in \cite{N}) and by Theorem \ref{reduced_volume}, we see that Perelman's reduced volume (\ref{flat_r_volume}) on the Ricci flat manifold $(M^n,g)$ is strictly decreasing under the static backward Ricci flow unless $(M^n,g)$ is isometric to the Euclidean space (also see Lemma 8.4 in \cite{CCGGI}).
Then it gives a light on that we may study the Ricci flat manifolds using the Perelman's reduced volume for Ricci flow.

The organization of the paper is as follows. In section 2, we recall some basic facts of Ricci flow, which will be used in the proof of Theorem \ref{main}. In section 3, we give the proof of Theorem \ref{main}.

\section{preliminaries}
In this section, we recall some basic facts of Ricci flow, which will be used in the proof of Theorem \ref{main}.
First, we need the following gradient estimate for Ricci flow proved by W.X.Shi \cite{S}.

\begin{thm}\label{shi_estimates}\cite{S}
Let  $(M^n, g(t))$ ($0\leq t\leq T$) be the solution to the Ricci
flow on the $n$-dimensional complete
manifold. Fix $p\in M$. Let $K<\infty$ and $\alpha>0$ be positive
constants. Then for each non-negative integer $k\leq \alpha$ and
each $r>0$, there is an uniform constant $C_k=C_k(K,r,n,\alpha)$
such that if
\begin{align*}
    |Rm(x,t)|\leq K, \ \text{for}\ (x,t)\in B_{g(0)}(p,r)\times [0,T],
\end{align*}
we have
\begin{align*}
    |\nabla^k Rm(y,t)|\leq \frac{C_k}{t^{\frac{k}{2}}}
\end{align*}
for all $y\in B_{g(0)}(p,\frac{r}{2})$ and $t\in (0,T]$.
\end{thm}

In \cite{P1}, Perelman proved the reduced volume defined in (\ref{r_volume}) has the following properties (also see \cite{KL} and \cite{MT}):

\begin{thm}\label{reduced_volume}\cite{P1}
Fix some point $p\in M^n$.
The reduced volume defined in (\ref{r_volume}) is
monotone non-increasing under the backward Ricci flow (\ref{backward_Ricci_flow}) and
$\mathcal {V}_p(\tau)\leq \lim\limits_{\tau\to 0}\mathcal
{V}_p(\tau)=1$. If $\mathcal
{V}_p(\tau_1)=\mathcal {V}_p(\tau_2)$ for some
$0<\tau_1<\tau_2$, then the Ricci flow is a gradient shrinking soliton.
If
 $\mathcal {V}_p(\bar{\tau})=1$ for some
time $\bar{\tau}>0$, then the Ricci flow is the trivial flow on flat
Euclidean space. Moreover, $\mathcal {V}_p^{i}(\tau)=\mathcal {V}_p(\lambda_i^{-1}\tau)$ under the rescaling
$g_i(\tau)=\lambda_i g(\lambda_i^{-1}\tau)$, where $\mathcal {V}_p^{i}(\tau)$ is the reduced volume of $(M^n,g_i(\tau))$.
\end{thm}

Finally, we need following point picking lemma used by Perelman \cite{P1} (also see Corollary 9.38 in \cite{MT} or \cite{KL}).
\begin{thm}\label{point_picking}
Let $(M, g)$ be a Riemannian manifold and let $p \in M$
and $r>0$ be given. Suppose that $B_g(p,2r)$ has compact closure in $M$ and
suppose that $f:B_g(p,2r)\to \mathbb{R}^+$ is a continuous, bounded function with $f(p) >
0$. Then there is a point $q\in B_g(p,2r)$ with the following properties:

(1) $f(q)\geq f(p)$;

(2) Setting $\alpha = \frac{f(p)}{f(q)}$, we have $d_g(p,q)\leq 2r(1-\alpha)$ and $f(q')<2f(q)$
for all $q'\in B_g(q,\alpha r)$.
\end{thm}

\section{proof of the main thoerem}

Before present the proof of Theorem \ref{main}, we need the following
\begin{lem}\label{lemma}
Let $(M^n,g)$ be the noncompact complete Ricci flat manifold with maximal volume growth. Suppose that $|Rm|(x)\to 0$ as $d(x)=d_g(x,p)\to \infty$. Then for any sequence $\{q_i\}$ with $d(q_i)\to \infty$, there exists a subsequence $\{q_i'\}$ of $\{q_i\}$ such that $\lim\limits_{d(q_i')\to\infty,\tau\to\infty}\mathcal {V}_{q_i'}(\tau)=1$.
\end{lem}

\begin{proof}
We take the subsequence $\{q_i'\}$ of $\{q_i\}$ as the following way. Let $q_1'=q_1$, $d_1=d_g(p,q_1')$, $B_1=B_g(q_1',\frac{d_1}{2})$. Since $d(q_i)\to \infty$, we can take $q_{i+1}'\notin B_g(p,4d_i)$, $d_{i+1}=d_g(p,q_{i+1}')$, $B_{i+1}=B_g(q_{i+1}',\frac{d_{i+1}}{2})$ for $i\geq 2$. Then we get a sequence of balls $\{B_i\}$ such that $B_j\cap B_k=\emptyset$ for $j\neq k$ and $B_{i+1}\subset B_g(p,2d_{i+1})\backslash B_g(p,2d_i)$.
Hence we have $\max\limits_{x\in B_i}|Rm|(x)\to 0$ as $i\to\infty$.
Moreover, the radius of the ball $B_i$ goes to infinity as $i\to \infty$ since $d_{i+1}\geq 4d_i$.
Note that $(M^n,g)$ has the maximal volume growth and $\max\limits_{x\in B_i}|Rm|(x)\to 0$ as $i\to\infty$.
Then the injectivity radius of $q_i'$ goes to infinity as $i\to\infty$ by Cheeger, Gromov and Taylor's injectivity radius
estimate (see \cite{CGT}).

Let $\Sigma_{q_i'}$ be the segment domain of $q_i'$ and take the normal coordinates on $\Sigma_{q_i'}$. Set $G_{q_i'}$ be the volume density with respect to the normal coordinates on $\Sigma_{q_i'}$.
Then we calculate
\begin{align*}
&|\mathcal {V}_{q_i'}(\tau)-1|\\
=&|\int_{M} (4\pi\tau)^{-\frac{n}{2}}e^{-\frac{d_g(q_i',y)^2}{4\tau}}dvol_{g}(y)-1|\\
=&|\int_{\Sigma_{q_i'}} (4\pi\tau)^{-\frac{n}{2}}e^{-\frac{|y|^2}{4\tau}}\sqrt{G_{q_i'}} dy^n-\int_{T_{q_i'}M} (4\pi\tau)^{-\frac{n}{2}}e^{-\frac{|y|^2}{4\tau}} dy^n|\\
=&|\int_{T_{q_i'}M} (4\pi\tau)^{-\frac{n}{2}}e^{-\frac{|y|^2}{4\tau}}(\chi(\Sigma_{q_i'})\sqrt{G_{q_i'}}-1) dy^n|\\
\leq& \sup\limits_{T_{q_i'}M}|\chi(\Sigma_{q_i'})\sqrt{G_{q_i'}}-1)|,
\end{align*}
where $\chi(\Sigma_{q_i'})(y)=\left\{
\begin{array}{l}
         1, \ \ \ y\in\Sigma_{q_i'}, \\
          0, \ \ \ y\notin\Sigma_{q_i'}.
\end{array}
\right.$
Since the injectivity radius of $q_i'$ goes to infinity, the radius of $B_i$ goes to infinity and $\max\limits_{x\in B_i}|Rm|(x)\to 0$, we conclude that
$\Sigma_{q_i'} \to T_{q_i'}M$ and $G_{q_i'}\to 1$ as $i\to\infty$. Then Lemma \ref{lemma} holds immediately.
\end{proof}

Now we can give the proof of Theorem \ref{main}.

\textbf{Proof of Theorem \ref{main}.}

We argue by contradiction. Suppose that $\sup\limits_{M^n}d^2(x)|Rm|(x)=\infty$. Then we can find a sequence $x_i\in M^n$ such that $|Rm|(x_i)\to 0$ and
$d^2(x_i)|Rm|(x_i)\to\infty$, where $d(x_i)=d_g(x_i,p)\to \infty$. Without loss of generality, we may assume $|Rm|(x_i)>0$ for all $i$. Then we consider the Ricci flow $\frac{\partial g}{\partial t}=-2 Rc$ on $(M^n,g)$ with $g(0)=g$. Let $\tau=-t$. Since $(M^n,g)$ is Ricci flat, the backward Ricci flow $\frac{\partial g}{\partial \tau}=2Rc$ on $M^n$ is static and exists for $(-\infty,0]$. We set $d_i=d_g(x_i,p)$, $B_i=B_{g}(x_i,\frac{d_i}{2})$ and $f=|Rm|^{\frac{1}{2}}(\cdot,0)$.

Now we use the technique of point picking due to Perelman \cite{P1}. Applying Theorem \ref{point_picking} to $B_i$ and $f$, we obtain that
there exists $y_i\in B_i$ such that
$$|Rm|(y_i,0)\geq |Rm|(x_i,0)$$
and
$$|Rm|(y,0)\leq 4|Rm|(y_i,0)$$
for all $y\in B_g(y_i,\frac{d_i|Rm|^{\frac{1}{2}}(x_i,0)}{2|Rm|^{\frac{1}{2}}(y_i,0)})\subset B_g(x_i,\frac{d_i}{2})$. Note that we still have
$d^2_g(p,y_i)|Rm|(y_i,0)\to \infty$ as $i\to\infty$ since $d_g(p,y_i)\geq \frac{d_i}{2}$. Then by Lemma \ref{lemma}, we conclude that there exists a subsequence $\{y_i'\}$ of $\{y_i\}$ such that $\lim\limits_{d_g(p,y_i')\to\infty,\tau\to\infty}\mathcal {V}_{y_i'}(\tau)=1$, where $\mathcal{V}_{y_i'}(\tau)$ is the reduced volume of $(M^n,g(\tau))$ based on point $y_i'$.

We consider the rescaled backward Ricci flow $g_i(\tau)=\lambda_ig(\lambda_i^{-1}\tau)$, where $\lambda_i=|Rm|(y_i',0)$. Then we see that $|Rm_{g_i(0)}|(y,0)\leq 4$ for all $y\in B_{g_i(0)}(y_i',\frac{d_i|Rm|^{\frac{1}{2}}(x_i,0)}{2})$. Since the solution to backward Ricci flow is static, we have $|Rm_{g_i(\tau)}|(y,\tau)\leq 4$ for all $y\in B_{g_i(0)}(y_i',\frac{d_i|Rm|^{\frac{1}{2}}(x_i,0)}{2})$. By Theorem \ref{shi_estimates}, all the derivatives of curvature $Rm(y,\tau)$ are uniformly bounded on $B_{g_i(0)}(y_i',\frac{d_i|Rm|^{\frac{1}{2}}(x_i,0)}{2})$ for any $\tau\in (-\infty,0)$.
Then we conclude that all the derivatives of curvature $Rm(y,\tau)$ are uniformly bounded on $B_{g_i(0)}(y_i',\frac{d_i|Rm|^{\frac{1}{2}}(x_i,0)}{2})$ for any $\tau\in (-\infty,0]$ since the backward Ricci flow on $M^n$ is static. Since $(M^n,g)$ has the maximal growth of volume, we obtain
$\frac{\text{Vol}_gB_g(y_i',r)}{r^n}\geq \Omega>0$ for all $r>0$. It follows that $\frac{\text{Vol}_{g_i(0)}B_{g_i(0)}(y_i',r)}{r^n}\geq \Omega>0$ by the Bishop-Gromov volume comparison theorem. Then the injectivity radius of $y_i'$ with respect to the metric $g_i(0)$ has the uniformly lower bound (see \cite{CGT}). By the fact $d_i|Rm|^{\frac{1}{2}}(x_i,0)\to \infty$ and Hamilton's precompactness theorem for the Ricci flow, we get there exists a sequence of $(M^n,g_i(\tau),y_i')$ (we still denote it by $(M^n,g_i(\tau),y_i')$ for simplicity) converges to $(M_{\infty}^n,g_{\infty}(\tau),y_{\infty})$ in $C^{\infty}$ sense, where $(M_{\infty}^n,g_{\infty}(\tau))$ is Ricci flat. Note that $|Rm_{g_i}|(y_i',0)=1$. It follows that $|Rm_{g_{\infty}}|(y_{\infty},0)=1$.

Then we consider the Perelman's reduced volume of $(M_{\infty}^n,g_{\infty}(\tau))$. We denote $\mathcal{V}^{\infty}_{y_{\infty}}(\tau)$  be the reduced volume of $(M_{\infty}^n,g_{\infty}(\tau))$ based on point $y_{\infty}$ and $\mathcal{V}^{i}_{y_i'}(\tau)$ be the reduced volume of $(M^n,g_i(\tau))$ based on point $y_i'$.
Fix $\tau>0$. The volume comparison implies
$$
\int_{d_{g_{\infty}}(y_{\infty},x)>A}(4\pi\tau)^{-\frac{n}{2}}e^{-\frac{d_{g_{\infty}}(x,y_{\infty})^2}{4\tau}}dvol_{g_{\infty}}(x)\leq C_1 e^{-C_2A},
$$
and
$$
\int_{d_{g_i}(y_i',x)>A}(4\pi\tau)^{-\frac{n}{2}}e^{-\frac{d_{g_i}(x,y_i')^2}{4\tau}}dvol_{g_i}(x)\leq C_1 e^{-C_2A},
$$
where $C_1$ and $C_2$ are the positive constants only depending on $n$ and $\tau$. Since $(M^n,g_i(\tau),y_i')$ converges to $(M_{\infty}^n,g_{\infty}(\tau),y_{\infty})$ in pointed Gromov-Hausdorff topology, we get
$$
\int_{d_{g_i}(y_i',x)\leq A}(4\pi\tau)^{-\frac{n}{2}}e^{-\frac{d_{g_i}(x,y_i')^2}{4\tau}}dvol_{g_i}(x)\to \int_{d_{g_{\infty}}(y_{\infty},x)\leq A}(4\pi\tau)^{-\frac{n}{2}}e^{-\frac{d_{g_{\infty}}(x,y_{\infty})^2}{4\tau}}dvol_{g_{\infty}}(x)
$$
for any $A>0$.
Hence
$$
\limsup\limits_{i\to\infty}|\mathcal
{V}^i_{y_i'}(\tau)-\mathcal{V}^{\infty}_{y_{\infty}}(\tau)|\leq 2 C_1 e^{-C_2A}.
$$
Taking $A\to \infty$, we conclude that
$$
\lim\limits_{i\to\infty}\mathcal
{V}^i_{y_i'}(\tau)=\mathcal{V}^{\infty}_{y_{\infty}}(\tau).
$$
Then by Theorem \ref{reduced_volume} and Lemma \ref{lemma}, we have
\begin{align*}
\mathcal{V}^{\infty}_{y_{\infty}}(\tau)
= \lim\limits_{i\to\infty}\mathcal
{V}^i_{y_i'}(\tau)
 =\lim\limits_{i\to\infty}\mathcal
{V}_{y_i'}(\lambda_i^{-1}\tau)
 =1.
\end{align*}
Then we conclude that $(M^n_{\infty},g_{\infty}(0))$ is isometric to $\mathbb{R}^n$ by Theorem \ref{reduced_volume}, which
contradicts to $|Rm_{g_{\infty}}|(y_{\infty},0)=1$.


\begin{thebibliography}{99}

\bibitem{A}
U.Abresch, \emph{Lower curvature bounds, Toponogov¡¯s theorem and bounded topology I}, Ann. Sci. Ec.
Norm. Sup. 18 (1985) 651-670.

\bibitem{AKL}
M.T.Anderson, P.B.Kronheimer, C.LeBrun, \emph{Complete Ricci-flat k\"{a}hler manifolds of infinite topological type}, Commun. Math. Phys. 125(1989),637-642

\bibitem{BKN}
S.Bando, A.Kasue, H.Nakajima, \emph{On a construction of coordinates
at infinity on manifolds with fast curvature decay and maximal volume
growth}, Invent. Math. 97:2 (1989), 313-349.

\bibitem{CC}
J. Cheeger, T. Colding, \emph{Lower bounds on the Ricci curvature and the almost rigidity of warped
products}, Ann. of Math. 144 (1996) 189-237.

\bibitem{CGT}
J.Cheeger, M.Gromov, M.Taylor. \emph{Finite propagation
speed, kernel estimates for functions of the Laplace operator, and
the geometry of complete Riemannian manifolds.} J.Diffential.Geom.
17(1982), 15-53

\bibitem{CT1}
J.Cheeger, G.Tian, \emph{On the cone at infinity of Ricci flat manifolds with Euclidean volume growth and
quadratic curvature decay}, Invent. math. 118 (1994), 493-571

\bibitem{CT}
J.Cheeger, G.Tian, \emph{Curvature and injectivity radius estimates for
Einstein 4-manifolds}, J. Amer. Math. Soc. 19 (2006), 487-525.

\bibitem{DM}
X.Dai, L.Ma, \emph{Mass under Ricci flow},
Commun. Math. Phys., 274, 65-80 (2007).

\bibitem{G}
M.Gromov, \emph{Volume and bounded cohomology}, Publ.Math.IHES 56 (1982) 5-99

\bibitem{CCGGI}
B.Chow, S.C.Chu, D.Glickenstein, C. Guenther, J. Isenberg, T.
Ivey, D.Knopf, P.Lu, F.Luo, and L.Ni, \emph{The Ricci flow: techniques and applications.}
Part I, Mathematical Surveys and Monographs, vol. 135, American Mathematical Society,
Providence, RI, 2007, Geometric aspects.

\bibitem{GH}
G.Gibbons, S.Hawking,  \emph{Gravitational multi-instantons}. Phys. Lett. 78B, 430-432 (1978)

\bibitem{GNP}
G.W.Gibbons, D.N.Page, C.N.Pope, \emph{Einstein metrics on $S^3$, $\mathbb{R}^3$ and $\mathbb{R^4}$ bundles}, Commun. Math. Phys.,
Volume 127, Number 3, 529-553

\bibitem{H}
N. Hitchin, \emph{Polygons and Gravitons}, Math. Proc. Cam. Phil. Soc. 85 (1979) 465-476.

\bibitem{KL}
B.Kleiner, J.Lott, \emph{Notes on Perelman's papers}, Geom. Topol. 12 (2008) 2587-2855

\bibitem{K}
P. Kronheimer,  \emph{The construction of ALE spaces as hyper-K\"{a}hler quotients}, J. Diff. Geom. 28 (1989) 665.

\bibitem{L1}
J.Lott, \emph{On the long time behavior of type III Ricci flow
solutions}, Math. Ann. 339(2007), 627-666.

\bibitem{L2}

J. Lott, Z. Shen, \emph{Manifolds with quadratic curvature decay and slow
volume growth}, Ann. Sci. Ecole Norm. Sup. (4) 33:2 (2000), 275-290.

\bibitem{M}

V.Minerbe, \emph{Weighted Sobolev inequalities and Ricci flat manifolds}, Geom.funct.anal. 18(2009) 1696-1749

\bibitem{MT}
J. Morgan, G.Tian, \emph{Ricci flow and the Poincar$\acute{e}$
conjecture}, Clay Mathematics Monographs Volume 3, American
Mathematical Society, Providence, RI, USA, 2007.

\bibitem{N}
A.Naber, \emph{Noncompact shrinking 4-solitons with nonnegative
curvature}, http://arxiv.org/abs/0710.5579v2.




\bibitem{RH3} R.Hamilton, \emph{Three-manifolds with positive Ricci curvature}.
 J. Differential Geom., 2(1982)255-306.

\bibitem{Pe1}
P. Petersen, \emph{Convergence theorems in Riemannian geometry}, in "Comparison
Geometry" (K. Grove, P. Petersen, eds.), MSRI Publications vol.
30 (1997), Cambridge Univ. Press, 167-202.

\bibitem{Pe2}
P. Petersen, \emph{Riemannian Geometry}. Springer Verlag GTM 171 (1997)

\bibitem{P1} G.Perelman, \emph{The entropy formula for the Ricci flow and its
geometric applications.} http://arxiv.org/abs/math/0211159.

\bibitem{P2} G.Perelman, \emph{A complete Riemannian manifold of positive Ricci curvature with Eucleadean volume growth and nonunique asymptotic cone.} Comparison geometry(Berkeley, CA, 1993-94), 165-166, Math.Sci.Res.Inst.Publ.,30,Cambrige Univ. Press, Cambrige, 1997.


\bibitem{SS}
J. Sha, Z. Shen, \emph{Complete manifolds with nonnegative Ricci curvature and quadratically nonnegatively
curved infinity}, Amer. J. Math. 119 (1997) 1399-1404.

\bibitem{S}
W.X.Shi, Ricci deformation of the metric on complete noncompact
Riemannian manifolds ,J.Diff. Geom., 30(1989)303-394.

\bibitem{TY}
G.Tian, S.T.Yau, \emph{Complete K\"{a}hler manifolds with zero Ricci curvature, I}.
J. Amer. Math. Soc., 3 (1990), 579-610

\bibitem{Y}
R.Ye, \emph{On the l-Function and the Reduced Volume of
Perelman II}, http://arxiv.org/abs/math/0609321v2


\end{thebibliography}
\end{document}